\documentclass[11pt, reqno]{amsart}
\usepackage{comment}

\usepackage{amsmath, amsthm, amssymb, mathrsfs, geometry, graphicx, tikz-cd, tikz}
\usetikzlibrary{decorations.pathreplacing}
\usepackage{hyperref}
\usepackage{booktabs}
\usepackage{placeins}
\usetikzlibrary{patterns}
\usetikzlibrary{calc}
\newtheorem{theorem}{Theorem}[section]
\newtheorem{lemma}[theorem]{Lemma}
\newtheorem{proposition}[theorem]{Proposition}

\theoremstyle{definition}
\newtheorem{definition}[theorem]{Definition}

\newtheorem{remark}[theorem]{Remark}

\title[HMS for Affine Log CY Surfaces]{Homological Mirror Symmetry for Affine Log Calabi-Yau Surfaces}
\author{Umut Varolg\"une\c{s}}

\address{
Department of Mathematics,
Ko\c{c} University,
Istanbul, T\"urkiye
}

\email{uvarolgunes@ku.edu.tr}

\begin{document}

\maketitle

\begin{abstract}
Let $U$ be a smooth complex affine log Calabi--Yau surface, and let
$\widehat{U}$ be its complete finite-type Liouville manifold, equipped
with its natural grading structure. We give an algorithm that constructs
a finite-type quasi-projective $\mathbb{Z}$-scheme $U^\vee$ such that
\[
D^\pi\bigl(\mathcal{W}(\widehat{U},\Bbbk)\bigr)
\cong
D^b\operatorname{Coh}(U^\vee_{\Bbbk})
\]
for every field $\Bbbk$. Our main contribution is to construct an almost
toric model encoded by an exact eigenray diagram and, using the
symplectic Torelli theorem for symplectic log Calabi--Yau pairs, to prove that this
model is grading-preserving strongly exact symplectomorphic to
$\widehat{U}$. The result then follows from the homological mirror
symmetry theorem of Hacking--Keating.
\end{abstract}

\section{Introduction}

Let $U$ be a (complex) affine log Calabi-Yau variety (resp. surface). This means that there exists a pair $(Y,D)$ with $Y$ a smooth projective complex variety (resp. surface), $D$ is a reduced and effective anticanonical divisor with at worst normal crossings singularities and $U\simeq Y\setminus D$ is affine. Let us call such pairs $(Y,D)$ log Calabi-Yau. We do not require $D$ to be maximally degenerate. 

In the surface case, we say that \begin{enumerate}
    \item $U$ is a Looijenga interior, if in any log CY compactification $(Y,D)$, $D$ has at least one node.
    \item $U$ is elliptic, otherwise (meaning $D$ is smooth).
\end{enumerate} It is well-known that these two cases have no overlap.

By classical results of Eliashberg-Gromov \cite{EG91} (see also \cite{CE}) $U$ defines a complete and finite-type Liouville manifold $(\widehat{U},\theta)$ well defined up to strongly exact symplectomorphism (which means a diffeomorphism $\phi$ such that $\phi^*\theta'=\theta+df$ with $f$ compactly supported). We grade $\widehat{U}$ using a holomorphic volume form on $U$ that has poles along the boundary of a log CY compactification (unique up to multiplication by a non-zero complex number). We can therefore define the $\mathbb{Z}$-graded wrapped Fukaya category $\mathcal{W}(\widehat{U},\Bbbk)$ where $\Bbbk$ is an arbitrary commutative ring. This category depends only on the deformation type of $U$ as in the following structural result, which is also standard.

\begin{theorem}
    Let $B$ be a connected smooth complex manifold. Let $\pi: \mathcal{Y} \to B$ be a proper holomorphic submersion. Let $\mathcal{D} \subset \mathcal{Y}$ be a divisor such that for each $b \in B$, the pair $(Y_b, D_b = \mathcal{D} \cap Y_b)$ is log Calabi-Yau. Assume that $\pi$ restricts to a submersion on each stratum of the stratification induced by the singularities of $\mathcal{D}$, and that each complement $X_b = Y_b \setminus D_b$ is an affine variety. Then, for any $b_0, b_1 \in B$: \begin{itemize}
        \item  The Stein manifolds $X_{b_0}$ and $X_{b_1}$ are Weinstein homotopic, and hence their Liouville completions $\widehat{X}_{b_0}$ and $\widehat{X}_{b_1}$ are strongly exact symplectomorphic.
        \item This exact symplectomorphism preserves the trivializations of the canonical bundles induced by the log Calabi-Yau condition up to homotopy.
        \item Consequently, there is an $A_\infty$-equivalence of $\mathbb{Z}$-graded wrapped Fukaya categories:$$\mathcal{W}(\widehat{X}_{b_0},\Bbbk) \simeq \mathcal{W}(\widehat{X}_{b_1},\Bbbk)$$
    \end{itemize}
\end{theorem}

We stress that this result is not actually used in any essential way in this paper - we could avoid even stating it. On the other hand, we felt like it might help some of our algebro-geometrically minded readers to make better sense of the content that will follow. With that in mind, we continue our formalization of the deformation invariance concept.

\begin{definition} If $U_0$ and $U_1$ are affine log Calabi-Yau varieties such that $U_i$ is isomorphic to $X_{b_i}$ for some $X_{b_0}$ and $X_{b_1}$ as above, then we call $U_0$ and $U_1$ deformation equivalent affine log Calabi-Yau varieties.
\end{definition}

Given a cohomologically unital $A_\infty$-category $\mathcal{A},$ we denote by $D^\pi(\mathcal{A})$ the triangulated category $H^0(\Pi(\mathrm{Tw}(\mathcal{A}))),$ where $\mathrm{Tw}$ denotes triangulated closure, $\Pi$ denotes idempotent closure (we could equivalently take $\mathrm{Perf}(\mathcal{A})$ instead of this two-step construction) and $H^0$ means taking the (triangulated) homotopy category associated to a pretriangulated $A_\infty$-category. If $U_1$ and $U_2$ are deformation equivalent affine log Calabi-Yau varieties, then there is an equivalence of triangulated categories: 
\[ D^\pi(\mathcal{W}(\widehat{U_1},\Bbbk)) \cong D^\pi(\mathcal{W}(\widehat{U_2},\Bbbk)). \]

\begin{definition}
    We say that a quasi-projective algebraic variety $U_\Bbbk^\vee$ over $\Bbbk$ is a $\Bbbk$-mirror of ${U}$ if there is a $\Bbbk$-linear exact equivalence of triangulated categories:
\[ D^\pi(\mathcal{W}(\widehat{U},\Bbbk)) \cong D^b \operatorname{Coh}(U^\vee_\Bbbk). \]
\end{definition}

Here is what we will prove in this paper.

\begin{theorem}\label{thm-main} Let $U$ be an affine log Calabi-Yau surface.
    There is a finite-type $\mathbb{Z}$-scheme $U^\vee$ that can be constructed algorithmically from $U$ whose base change to $\mathrm{Spec}(\Bbbk)$  is a $\Bbbk$-mirror of $U$ for $\Bbbk$ an arbitrary field.

Here is the algorithm.

\textbf{Case 1, $U$ is a Looijenga interior:} We first construct a toric model of $U$ (\cite{GHK15}, see the beginning of Section \ref{sec-toric-model} for our specific terminology). To construct the mirror $U^\vee$ we take the fan of the Toric Start from the previous sentence, which defines a finite-type $\mathbb{Z}$-scheme, and we make the same number of non-toric blowups on each boundary divisor but crucially always at the special point $-1\in\mathbb{G}_m$ iteratively. $U^\vee$ is obtained by removing the boundary divisor. 

\textbf{Case 2, $U$ is elliptic:} then it must have a log CY compactification $(Y,D)$ where $Y$ is a del Pezzo surface and $D$ is a smooth anticanonical divisor (an elliptic curve). Therefore, up to deformation equivalence there are $10$ elliptic affine log CY surfaces. 
\begin{enumerate}
    \item If $U= \mathbb{P}^1\times \mathbb{P}^1\setminus V_{2,2},$ with $V_{2,2}$ the vanishing locus of a generic bidegree $(2,2)$ polynomial, $U^\vee$ is obtained by doing a toric blowup at each corner of  $\mathbb{P}^1\times \mathbb{P}^1$, then doing a non-toric blow-up at $-1$ in each exceptional divisor and removing the boundary.
    \item Otherwise, we can distinguish the $9$ cases by looking at the self-intersection number $D^2$ and define $d=D^2\in \{1,\ldots,9\}.$ To construct the mirror, take the complete fan with rays  $(1, 0)$, $(7, 1)$, $(-41, -5)$, and $(2d - 17, -2)$ and do $9-d,1,1$ and $1$ non-toric blowups at $-1$ respectively. The complement of the boundary divisor gives our mirror finite-type $\mathbb{Z}$-scheme in this case. 
\end{enumerate}

\end{theorem}

\begin{remark} The only computation in Case 1 is to find a toric model of $U$. The difficulty of this depends on how the affine log CY is presented to us. Most naively, the process involves finding explicit $-1$ curves on surfaces and contracting them, which may be computationally expensive (though theoretically algorithmic). On the other hand, there appear to be better alternatives. Not being an expert, we avoid further discussion and assume that a toric model can be effectively found.

For the Painlevé character varieties, which form a distinguished (but finite and small) family of examples covered by Theorem \ref{thm-main} Case 1, Beimler, Hu, Olsen, and Shende~\cite{BHOS} also establish homological mirror symmetry, by different methods, and identify the mirror with the minimal resolution of the corresponding unipotent Painlevé character variety.

\end{remark}
    
\begin{remark}    
    Hacking--Keating \cite[Section~6.3]{HK22}, using Vianna's work
\cite{Via17}, give another algorithm for Case~2. Here we give a more
uniform approach.

        Since the result is computationally heavy, it is worth doing the following sanity check. The resulting primitive vectors $(1, 0)$, $(7, 1)$, $(-41, -5)$, and $(2d - 17, -2)$ define our B-side toric fan. Following the algorithm, we execute $k = 9-d$ non-toric blowups on the divisor of $(1,0)$, and $1$ non-toric blowup on each of the other three boundary divisors. The complement of the boundary divisor gives our mirror finite-type $\mathbb{Z}$-scheme in this case. Resolving the complete fan via continued fractions and iteratively contracting all boundary $(-1)$-curves yields a blowdown cascade. As we show in the Appendix, for any $k \in \{0, \dots, 8\}$, the boundary becomes a cycle of exactly $9-k$ curves (irreducible but nodal for $k=8$), each with a self-intersection of $-2$ for $k\neq 8$ and with self-intersection $0$ if $k=8$. With a little more work, we show that indeed, the base change to complex numbers is a rational elliptic surface minus an $I_{d}$-fiber (compare for example with \cite{AKO06}).
\end{remark}

\subsection{Hacking--Keating Homological Mirror Symmetry}

We will prove Theorem \ref{thm-main} by showing that it is actually a special case of the main result of \cite{HK22}.

\begin{definition}
    Let us call a finite collection of rational slope rays in the Euclidean plane $\mathbb{R}^2$ starting at the origin, each equipped with a non-negative integer, a rational ray diagram. 
\end{definition}

Given a rational ray diagram $\mathcal{R}$, we can define a Liouville manifold $U_\mathcal{R}$ (well-defined up to strongly exact symplectomorphism) and a finite-type $\mathbb{Z}$-scheme $U^\vee_\mathcal{R}$. 

\subsubsection{Liouville manifolds from rational ray diagrams}
First, on each ray of $\mathcal{R},$ we choose its multiplicity many distinct points that are distinct from the origin. These are the nodes and we can equip the complement of the nodes in the plane with an integral affine structure. The complement of the union of subrays emanating from the nodes with its natural inclusion is an affine chart called the reference chart. This is extended to an integral affine atlas such that the affine monodromy around each node is a unipotent shear transformation by the primitive integral vector along the ray. The given subrays become eigenrays of their respective nodes. The signs are so that when a geodesic intersecting one of these eigenrays is drawn in the reference chart, as it passes through the eigenray, it always bends towards the node. Moreover, the more of these eigenrays the geodesic intersects, the sharper the bend. We call the nodal integral affine manifold $B_\mathcal{R}$. Note that all nodes lie on rays from the origin with those rays as invariant lines. 

\begin{remark}
    In \cite{GV23}, we defined an \emph{eigenray diagram} to consist of finitely many disjoint rational-slope rays in $\mathbb{R}^{2}$, together with finitely many marked points of positive integral multiplicity encoding the focus--focus nodes. It is called \emph{exact} if the affine lines containing all of its rays pass through a common point $b_{0}$ \cite[Remark~1.10]{GV23}. With this language in mind, we notice that $B_\mathcal{R}$ comes from an exact eigenray diagram.  
\end{remark}

We then construct a 4-dimensional symplectic manifold $M_{\mathcal{R}}$ equipped with a Lagrangian torus fibration $\pi: M_{\mathcal{R}} \to B_{\mathcal{R}}$. Over the regular locus $B_{\mathcal{R}}^{\text{reg}}$, the manifold is symplectomorphic to the standard action-angle bundle $M_{\mathcal{R}}^{\text{reg}}:=T^*B_{\mathcal{R}}^{\text{reg}} / \text{lattice}$. Standard neighborhood models of focus-focus singularities are then glued in over the nodes $\{b_i\}_{i \in I}$.

We now consider the \emph{radial vector field} $\sum_{j=1}^2 x_j \frac{\partial}{\partial x_j}$ in the reference chart, which extends uniquely to a smooth vector field on $B_{\mathcal{R}}^{\text{reg}}$. Any such vector field on the base lifts uniquely to a horizontal (with respect to the flat horizontal subbundle induced by the integral-affine structure) vector field on $M_{\mathcal{R}}^{\text{reg}}$. A simple computation shows that the horizontal lift of the radial vector field is a Liouville vector field on the total space.

This lift is undefined over the nodes but a relative de-Rham cohomology argument shows the existence of a Liouville vector field $Z$ that agrees with the horizontal lift of the radial vector field outside of arbitrarily small neighborhoods of the nodal tori.

\begin{proposition}\label{prop-R-Liouville-eq}
    Different choices made in the construction of $U_\mathcal{R}$ lead to strongly exact symplectomorphic complete and finite-type Liouville manifolds. 
\end{proposition}

Finally, note that there is a unique grading structure on $U_\mathcal{R}$ which makes the Lagrangian torus fibers of $\pi_\mathcal{R}$ Maslov zero. We will equip $U_\mathcal{R}$ with this grading structure.

\subsubsection{Finite-type $\mathbb{Z}$-schemes from rational ray diagrams.}

We take the rank-two lattice $N:=\mathbb Z^2$, write $M=N^\vee$,
and choose a smooth complete fan $\Sigma$ in
$N_{\mathbb R}\cong\mathbb R^2$ whose one-dimensional cones include
the rays of $\mathcal R$. This defines a smooth toric
$\mathbb Z$-scheme $Y_\Sigma$. For each $\rho_j\in\Sigma(1)$, let
$D_j$ be the corresponding toric boundary divisor and let
$D_j^\circ=O(\rho_j)$ be its open torus orbit. Then
\[
D_j^\circ
=
\operatorname{Spec}\mathbb Z[M\cap\rho_j^\perp]
\cong\mathbb G_m.
\]
The rank-one lattice $M\cap\rho_j^\perp$ has two primitive generators;
the corresponding identifications of $D_j^\circ$ with $\mathbb G_m$
differ by inversion and hence determine the same distinguished section
$-1\in D_j^\circ(\mathbb Z)$.

    If the ray corresponding to $D_j$ has multiplicity $m_j>0$, we perform a sequence of $m_j$ infinitely near blow-ups: first we blow up the section
\[
p_{j,0}\colon \operatorname{Spec}\mathbb Z\longrightarrow D_j^\circ
\]
corresponding to $-1\in D_j^\circ(\mathbb Z)$; after the $\ell$-th blow-up, for $1\leq\ell<m_j$, we take the $(\ell+1)$-st center to be the intersection section of the strict transform of $D_j$ with the most recently created exceptional divisor. If $m_j=0$, no blow-up is performed.

 Let $\pi: \widetilde{Y} \longrightarrow Y_{\Sigma}$
be the resulting surface after performing all the blowups.
 Let $\widetilde{D}$ be the  strict transform of the original toric boundary divisor under the map $\pi$. The pair $(\widetilde{Y}, \widetilde{D})$ forms a log Calabi--Yau surface. Finally $U^\vee_\mathcal{R}$ is defined as $\widetilde{Y} \setminus \widetilde{D}$.

\subsubsection{Homological mirror symmetry for rational ray diagrams}

\begin{theorem}\label{thm-hk}
Let $\mathcal{R}$ be a rational ray diagram. Then, the base change $U^\vee_{\mathcal{R}, \Bbbk}$ of $U^\vee_\mathcal{R}$ to an arbitrary field $\Bbbk$ is a $\Bbbk$-mirror of $U_\mathcal{R}$:
\[ D^\pi(\mathcal{W}(U_\mathcal{R},\Bbbk)) \cong D^b \operatorname{Coh}(U^\vee_{\mathcal{R}, \Bbbk}). \]

\end{theorem}

\begin{proof}
This follows from the third bullet point of
\cite[Theorem~1.1]{HK22}, combined with
\cite[Theorem~1.2]{HK22}. Although
\cite[Theorem~1.1]{HK22} is stated over $\mathbb{C}$,
Hacking and Keating have confirmed that the proof works without
change with coefficients in an arbitrary field $\Bbbk$. We therefore
apply their argument over $\Bbbk$, with algebraic mirror given by the
base change $U_{\mathcal R}^{\vee}\otimes_{\mathbb Z}\Bbbk$.
\end{proof}
  
\begin{remark}
    A natural question is whether we also have \[ D^\pi(\mathcal{W}(U_\mathcal{R})) \cong D^b \operatorname{Coh}(U^\vee_{\mathcal{R}}). \] We expect this to be true and hope to address it in future work, but the existing arguments in the literature do not prove this.
\end{remark}

\subsection{Rational ray diagrams associated to $U$}

Let $U$ be an affine log CY surface. We say that a rational ray diagram $\mathcal{R}$ models the symplectic geometry of $U$ if $U_\mathcal{R}$ is strongly exact symplectomorphic to $\widehat{U}$ in a way that preserves the grading data.

The contribution of the paper is to prove the following result, which immediately implies Theorem \ref{thm-main} by Theorem \ref{thm-hk}.

\begin{theorem}\label{thm-ray} Let $U$ be an affine log Calabi-Yau surface. In each case, the given rational ray diagram models the symplectic geometry of $U.$

\textbf{Case 1, $U$ is a Looijenga interior:} Consider any toric model of $U$. Take the one-dimensional cones of the fan of the Toric Start as the rays and the number of non-toric blowups on the corresponding boundary divisor component as the multiplicities.

\pagebreak
\textbf{Case 2, $U$ is elliptic:} 
\begin{enumerate}
    \item If $U= \mathbb{P}^1\times \mathbb{P}^1\setminus V_{2,2},$ take the four rays in the diagonal and anti-diagonal directions with multiplicity $1.$
    \item Otherwise, if $D$ is the boundary in a log CY compactification with self-intersection number $d\in \{1,\ldots,9\},$  take the complete fan with rays  $(1, 0)$, $(7, 1)$, $(-41, -5)$, and $(2d - 17, -2)$ and with multiplicities $9-d,1,1$ and $1$ respectively.
\end{enumerate}

\end{theorem}

\paragraph{\textbf{Organization of the paper.}}
Section \ref{sec2} constructs the rational ray diagram associated to an affine log Calabi--Yau surface. After establishing the required generalized toric models and constructing suitable moment polygons and Symington triangles, we treat the Looijenga and elliptic cases separately and obtain the diagrams appearing in Theorem \ref{thm-ray}. We then compare the resulting almost toric models with the algebraic compactification using the Torelli theorem for symplectic log CY surfaces and prove that the obtained diagrams indeed model the symplectic geometry of the affine surfaces. This proves Theorem \ref{thm-ray} and hence Theorem \ref{thm-main}. The Appendix carries out the toric resolution, boundary blowdown, and period calculations used to identify the mirrors in the elliptic case as rational elliptic surfaces with an $I_d$-fiber removed.

\paragraph{\textbf{Acknowledgements.}} The author thanks Ailsa Keating for helpful email correspondence confirming that the Hacking--Keating homological mirror symmetry argument in \cite{HK22} works with coefficients in an arbitrary field. This work was supported by the Scientific and Technological Research Council of Türkiye (TÜBİTAK) through the 3501 Career Development Program (project no.\ 124F451). The author used large language models, including ChatGPT and Gemini, as auxiliary tools for proofreading, improving the exposition, and critically reviewing the arguments. The author takes full responsibility for all mathematical content and conclusions of the paper.

\section{Constructing a rational ray diagram associated to $U$}\label{sec2}

\subsection{Toric models of affine log CY surfaces}\label{sec-toric-model}

A \emph{generalized toric model} for an affine log CY surface $U$ is a log CY pair $(Y, D'')$ such that $U \cong Y \setminus D''$, constructed via the following geometric sequence (which is part of the data):

\textbf{Toric Start:} Start with a smooth projective toric surface $X$ and let $D_X$ be its toric boundary forming the log CY pair $(X, D_X)$.

\textbf{Nodal Smoothing:} Perform complex deformations of $D_X$ in the anticanonical linear system to smooth some chains, or the entire cycle, of toric boundary components $I_i \subset D_X$, yielding the log CY pair $(X,D')$ with a partially smoothed boundary $D' = \sum_{i=1}^m C_i$.

\textbf{Non-Toric Blowups:} Perform $k_i \ge 0$ non-toric blowups on the smooth open locus of each $C_i$ to yield the final surface $Y$, with strict transform boundary $D'' = \sum_{i=1}^m C''_i$.

If nothing is done in the Nodal Smoothing step, we call this a \emph{toric model}.

\begin{theorem}\label{thm-toric-model}

Every affine log CY surface $U$ admits a generalized toric model where $(X, D_X)$ satisfies the following:

\begin{enumerate}

\item If $D''$ is smooth, then $X$ is isomorphic to $\mathbb{P}^2$ or $\mathbb{P}^1\times \mathbb{P}^1$ (with $D'$ a smooth anticanonical curve, necessarily of genus $1$).

\item If $D''$ is not smooth, we have a toric model.

\end{enumerate}

\end{theorem}

\begin{proof}

Let $(Y,D'')$ be a log Calabi--Yau compactification such that $U=Y\setminus D''$ is affine. Since $Y$ is smooth, Goodman's theorem \cite[Theorem~2]{Goodman1969} implies that $D''$ supports an ample divisor. Let \(A\) be an ample divisor supported on \(D''\). Then
\[
K_Y\cdot A=-D''\cdot A<0.
\]
Consequently no positive multiple of \(K_Y\) is effective, and hence
\(\kappa(Y)=-\infty\). Hence $Y$ is either rational or ruled over a curve of positive genus. In the latter case, an elementary intersection calculation on the relatively minimal ruled model shows that no divisor supported on the anticanonical divisor can be ample. Therefore $Y$ is rational.

First, note that case (2) is a special case of \cite[Proposition 1.3]{GHK15}. 

Now, assume that $D''$ is smooth. We run the Minimal Model Program (MMP) on $Y$ contracting $(-1)$-curves. By the adjunction formula on the smooth surface $Y$, any $(-1)$-curve $E$ satisfies:
\[
-2 = 2p_a(E) - 2 = E^2 + K_Y \cdot E = -1 + K_Y \cdot E \implies -K_Y \cdot E = 1.
\]

Since $(Y, D'')$ is log CY, $D'' \sim -K_Y$, which forces $E \cdot D'' = 1$. This proves that every $(-1)$-curve is either a component of the boundary or intersects the boundary transversely exactly once. Since $D''$ is a smooth elliptic curve, the former is not possible. Contracting $E$ reverses a non-toric blowup on the boundary.

Iterating this process strictly decreases the Picard rank and terminates in a minimal rational log CY pair $(X, D')$ where the surface $X$ contains  no $(-1)$-curves and $D'$ is also a smooth elliptic curve. By the classification of minimal rational surfaces, $X$ must be isomorphic to either $\mathbb{P}^2$ or a Hirzebruch surface $\mathbb{F}_n$ with $n \neq 1$. 

We want to prove that $X=\mathbb{F}_n$ with $n\geq 2$ is not possible. Let $C_0$ be the unique negative curve with $C_0^2 = -n$, and let $f$ be the fiber class. The anticanonical class is $-K_X \sim 2C_0 + (n+2)f$.

We evaluate the intersection number of $D'$ with $C_0$:
\[
D' \cdot C_0 = (2C_0 + (n+2)f) \cdot C_0 = 2(-n) + (n+2) = 2 - n.
\]

Since $C_0$ cannot be an irreducible component of $D'$, we have $D' \cdot C_0\geq 0$. Therefore, $n\leq 2.$ In fact, if $D' \cdot C_0= 0$,  $C_0$ is disjoint from $D'$. Its proper transform in $Y$ would therefore be a complete curve
contained in the affine surface $U=Y\setminus D''$, which is
impossible. Therefore, $n=2$ is also not possible and we are done.

\end{proof}

\begin{proposition}
Let $U$ be an affine log CY surface which admits a generalized toric model with $X=\mathbb{P}^2$ or $\mathbb{P}^1\times \mathbb{P}^1$ such that the final boundary $D''$ is a smooth curve. Then the number of non-toric blowups $k$ on $D'$ satisfies the  bounds presented below.

\begin{table}[htbp]
\centering
\renewcommand{\arraystretch}{1.8}
\begin{tabular}{@{}ccllc@{}}
\toprule
\textbf{Surface} &  \textbf{Boundary} & \textbf{Self-intersection} & \textbf{Bound} \\
\midrule
$\mathbb{P}^2$ &Smooth Cubic & $9$ & $\displaystyle k \le 8$ \\
$\mathbb{P}^1\times \mathbb{P}^1$& Smooth $(2,2)$ curve & $8$ & $\displaystyle k \le 7$ \\
\bottomrule
\end{tabular}

\label{tab:smooth_bounds}
\end{table}

Moreover, for $k=2,\ldots, 8$, the affine varieties coming from $\mathbb{P}^2$ with $k$ blowups and those coming from $\mathbb{P}^1\times \mathbb{P}^1$ with $k-1$ blowups are the same. 
\end{proposition}

\begin{proof}
The first part follows immediately by the Nakai-Moishezon criterion using that each non-toric blow-up reduces self-intersection number by $1.$ Concerning the second part, for \(2\leq k\leq 8\), the classification of del Pezzo surfaces shows that the
same del Pezzo pairs admit presentations both as \(k\) boundary blowups of
\(\mathbb{P}^2\) and as \(k-1\) boundary blowups of
\(\mathbb{P}^1\times\mathbb{P}^1\). Thus the same log Calabi--Yau pairs \((Y,D'')\) admit both
presentations, and taking complements gives the same collection of
affine varieties.
\end{proof}

\subsection{Moment Polygons and Symington Triangles}

The main result of this section is a version of the Symington triangle packing technique used by \cite{EF21, LMN23, LMN25}. We use the notation from the previous section for the generalized toric model.

\begin{theorem}\label{thm-tri-pack}

Let $U$ be an affine log CY surface with a generalized toric model whose Toric Start $(X, D_X)$ is such that if a component of $D_X$ has negative self-intersection, then during the Nodal Smoothing step it is not deformed. Note that this condition holds for the generalized toric models produced by Theorem \ref{thm-toric-model}.

Then, there is a Kähler structure on $Y$, coming from a positive Hermitian metric on some integral ample divisor  $A = \sum_{i=1}^m \mu_i C''_i$ supported on the boundary $D''$, satisfying the following properties:

\begin{itemize}
    \item There exists a moment polytope $\Delta_X$ of $X$ with an interior point $p$ such that for every $i$, the affine distance from $p$ to each edge corresponding to a component $L\in I_i$ is equal to $\mu_i$ and the total affine length $l_i$ of these edges is strictly larger than $k_i\mu_i,$ where $k_i$ denotes the number of non-toric blow-ups.
    \item The difference $l_i-k_i\mu_i$ is equal to the symplectic area of $C_i''$. 
\end{itemize}

\end{theorem}

\begin{proof}

Because $U = Y \setminus D''$ is an affine log CY surface, $D''$ supports an \emph{effective} ample $\mathbb{Z}$-divisor by Goodman's theorem \cite[Theorem~2]{Goodman1969}. Let this ample divisor be $A = \sum_{i=1}^m \mu_i C''_i$ with $\mu_i>0$. 

Choose a positive Hermitian metric on $\mathcal{O}_Y(A)$, and let
$\omega_Y$ be its curvature form, normalized so that
\[
[\omega_Y]=c_1(\mathcal{O}_Y(A)).
\]
Thus, for every curve $C\subset Y$,
\[
\int_C\omega_Y=A\cdot C.
\]
In particular, if $E$ is an exceptional curve meeting $C_i''$ once and no other boundary component, then
\[
\int_E\omega_Y=A\cdot E=\mu_i.
\]

Since smoothing preserves the divisor class, one has
\[
[C_i]=\sum_{L\in I_i}[L]\in\operatorname{Pic}(X).
\]
Consider the following $\mathbb{Z}$-divisor on the toric surface $X$:
\[
H_X = \sum_{i=1}^m \mu_i \left( \sum_{L \in I_i} L \right),
\] where $C_i$ is obtained by smoothing the chain $I_i$ on $X$ prior to the non-toric blowups. We claim that $H_X$ is an ample divisor. It suffices to show that $H_X \cdot L > 0$ for every irreducible toric boundary curve $L \subset D_X$.

We first evaluate $H_X\cdot C_i$. The $k_i$ non-toric blowups drop the self-intersection of the strict transforms such that $(C''_i)^2 = C_i^2 - k_i$, while intersections between distinct boundary components remain unchanged. Evaluating the intersection of $A$ with $C''_i$ yields the fundamental identity:
\[
A \cdot C''_i = \mu_i(C_i^2 - k_i) + \sum_{j \neq i} \mu_j (C_j \cdot C_i) = \left( \sum_{j=1}^m \mu_j C_j \right) \cdot C_i - k_i \mu_i = H_X \cdot C_i - k_i \mu_i.
\]

Because $A$ is ample on $Y$, we have $A \cdot C''_i > 0$. We obtain the inequality:
\[
H_X \cdot C_i - k_i \mu_i > 0 \implies H_X \cdot C_i > k_i \mu_i \ge 0.
\]

Let us now go back to our claim about $L \subset D_X$. There are two options:

\begin{enumerate}

\item If $L$ is not merged with any other toric boundary curve, then $L = C_i$ for some $i$. By our fundamental inequality, $H_X \cdot L = H_X \cdot C_i > k_i \mu_i \ge 0$.

\item Suppose $L$ belongs to a chain $I_i$ of length $\ge 2$, which we are given implies that $L^2 \ge 0$. It intersects exactly two other toric boundary components in $D_X$, say $L'$ and $L''$. Let $L' \in I_{j'}$ and $L'' \in I_{j''}$, where $j'$ and $j''$ may or may not equal $i$. The intersection evaluates to:
\[
H_X \cdot L = \mu_{j'} (L' \cdot L) + \mu_{j''} (L'' \cdot L) + \mu_i L^2 = \mu_{j'} + \mu_{j''} + \mu_i L^2\ge \mu_{j'} + \mu_{j''} > 0.
\]

\end{enumerate}

Therefore, $H_X$ is indeed ample. Now consider the corresponding moment polytope:
\[ \Delta_X = \{ u \in \mathbb{R}^2 \mid \langle u, \nu_L \rangle + \mu_{i(L)} \ge 0 \text{ for all toric boundary components $L$} \}, \] where $\nu_L$ is the primitive generator of the one-dimensional cone of the fan corresponding to $L$ and $L$ is smoothed to $C_{i(L)}.$

We choose $p$ to be the origin, which is of affine distance $\mu_i$ to every toric boundary curve $L \in I_i$ by construction. Let $l_i$ be the sum of the affine lengths of the edges corresponding to $I_i$. By construction, we have $l_i = H_X \cdot C_i$, which shows $l_i > k_i \mu_i$ by the fundamental inequality, as desired.

Finally, using the normalization of $\omega_Y$ and the fundamental identity, we obtain
\[
\int_{C_i''}\omega_Y
=A\cdot C_i''
=H_X\cdot C_i-k_i\mu_i
=l_i-k_i\mu_i.
\]

\end{proof}

\subsection{Constructing a rational ray diagram associated to $U$}

In this section, we will construct a rational ray diagram associated to $U$ using the technique described on page 145 of \cite{Eva23}. Let's first review some definitions, referring the reader to \cite{Eva23} for more details and explanations.

\begin{definition}[Symington Triangle]
In an almost toric base, a \emph{Symington triangle} on a boundary edge $E$ (which can be the entire boundary) is an affinely embedded triangle $T$, disjoint from the nodes, whose intersection with the boundary is precisely one of its edges $B$ (the base edge), which is contained in the interior of $E$, and the integral affine distance of the apex of $T$ to $E$ measured by one of the side edges of $T$ is equal to the integral affine length of $B.$ If the integral affine distance is equal to the integral length divided by some positive integer $d$, we refer to it as a \emph{$d$-big Symington triangle}.
\end{definition}
\begin{definition}[Nodal Trade]
In an almost toric base, a \emph{nodal trade} at a boundary vertex $v$ is a local modification that replaces a standard corner neighborhood of $v$ with a neighborhood containing a smooth boundary segment and exactly one interior node. This operation is performed by excising a standard neighborhood of the corner and introducing a branch cut that extends from the new interior node to the boundary, directed along the invariant eigendirection of the node's affine monodromy, which smooths the corner into a straight edge in the affine structure.
\end{definition}

Here is a summary of the strategy.

\begin{enumerate}
    \item Case 1: We are given $\Delta_X$ as in Theorem \ref{thm-tri-pack} where nothing is done in the Nodal Smoothing step as guaranteed by Theorem \ref{thm-toric-model}. For each $i$ with $k_i>0$, we can embed one $k_i$-big
Symington triangle with apex $p$ and base on the $i$-th boundary
edge $e_i$. Any two of these big triangles intersect only at $p$. Then, we can do nodal slides as described below to perturb these and make them completely disjoint (cf. \cite[Lemma 32]{LMN23} and \cite[Proposition 5.6]{EF21}) and then do further nodal slides to separate the big Symington triangles into disjoint Symington triangles. We then apply the non-toric blowup operation along these triangles and use branch moves as described on page 145 of \cite{Eva23} to get our rational ray diagram. Please see Figure \ref{fig:332} for a summary. 
    \item Case 2(1): Consider  the $\mathcal{O}(2,2)$ moment polytope of $\mathbb{P}^1\times \mathbb{P}^1$ and do all nodal trades at the vertices.
    \item Case 2(2): We are given $\mathbb{P}^2$ with $\mathcal{O}(3)$  moment polytope. After nodal trades and some Vianna-type branch moves, we can embed at most $8$ disjoint Symington triangles with their base on the long edge and tip vertices all contained in the ray emanating from the monotone point $p$, parallel to the long edge of the boundary (with respect to a straight path from $p$ to the edge that lies in the reference chart). Please see Figures \ref{fig-Vianna} and \ref{fig-fan}.
\end{enumerate}

Let us now give more details.
\subsubsection{Case 1}
\label{sec:case1}

In Case 1, we are given a Delzant moment polytope $\Delta_X$ as in Theorem \ref{thm-tri-pack}, and a point $p$. Let $e_1, \dots, e_m$ be the edges of $\Delta_X$ ordered counterclockwise, with primitive counterclockwise direction vectors $\vec{v}_1, \dots, \vec{v}_m$. We must embed $k_i \ge 0$ standard Symington triangles on each boundary edge $e_i$ (where $i=1,\dots,m$) such that their integral affine base lengths exactly match the affine distance from their apex to the edge, and such that all $\sum k_i$ triangles are completely disjoint. We recommend the reader to follow the steps in the example described by Figure \ref{fig:332}.

We first embed one $k_i$-big Symington triangle per edge with tip at $p$ and base affine length $k_i$ times the affine distance of $p$ to that edge, for every $i$ with $k_i\geq 1$. If $k_i=0$, for the sake of uniformity of notation in the argument, we take a straight line segment drawn from $p$ to $e_i$ and think of it as one of these triangles. We can make sure that these triangles (called big triangles) only intersect at $p$. We get rid of that intersection by performing nodal slides parallel to the edges by the Resolving Triangle Packing Lemma right below. Once this is done by further nodal slides we can also separate out the $k_i$ standard Symington triangles for each edge.

\begin{lemma}[Resolving Triangle Packing]
For a sufficiently small $\epsilon > 0$, we define the shifted big triangle for each edge by shifting the entire big triangle parallel to $\vec{v}_i$ by $\epsilon$ times the inverse of the affine distance of $p$ to $e_i$. The apex of the shifted big triangle becomes
\begin{equation}
    A_i := p + \frac{\epsilon}{d(p, e_i)} \vec{v}_i.
\end{equation} 

Here sufficiently small only means that, after the shifts, triangles are still contained in $\Delta_X$ and the base edges are contained in the interiors of the $e_i$. 

Then, the shifted big triangles are automatically disjoint.
\end{lemma}

\begin{proof}
We first notice that the vector connecting $A_{i-1}$ to $A_i$ is a positive scalar multiple of the vector pointing from $R_i$ to $p$. One way to see this is to apply an integral affine transformation that maps the corner $R_i$ to the origin $(0,0)$. We can locally align $e_i$ with the positive $x$-axis and $e_{i-1}$ with the positive $y$-axis (meaning $\vec{v}_i = (1,0)$ and $\vec{v}_{i-1} = (0,-1)$).

Let $p = (x,y)$. By definition, $d(p, e_{i-1}) = x$ and $d(p, e_i) = y$. Evaluating the shifts yields:
\begin{align*}
    A_{i-1} &= p + \frac{\epsilon}{x} (0,-1) = \left(x, y - \frac{\epsilon}{x}\right) \\
    A_i &= p + \frac{\epsilon}{y} (1,0) = \left(x + \frac{\epsilon}{y}, y\right)
\end{align*}
The vector connecting $A_{i-1}$ to $A_i$ is calculated as:
\begin{equation}
    A_i - A_{i-1} = \left(\frac{\epsilon}{y}, \frac{\epsilon}{x}\right) = \frac{\epsilon}{xy} (x,y) = \frac{\epsilon}{d(p, e_{i-1}) d(p, e_i)} (p - R_i),
\end{equation} proving the desired claim.

Because $\Delta_X$ is strictly convex, the vectors $p - R_i$ rotate counterclockwise by less than $\pi$ radians as we traverse the corners. Consequently, the edges of the sequence $A_1, \dots, A_m$ also rotate counterclockwise by less than $\pi$ radians, proving that they form a \emph{strictly convex} polygon $P_{\text{micro}}$.

The extended edges of any strictly convex polygon partition the complement into a set of ``exterior wedges.'' Here by exterior wedge we specifically take those wedges bounded by the rays obtained by extending each edge $A_iA_{i-1}$ in the direction of $A_{i-1}$. We now notice that the  $i$-th shifted triangle is completely confined within the $i$-th exterior wedge of $P_{\text{micro}}$ and, except for the vertex $A_i$, is contained in the interior of that exterior wedge. Because the unions of the interiors with their vertices (of each wedge) are pairwise disjoint, the proof is complete.
\end{proof}

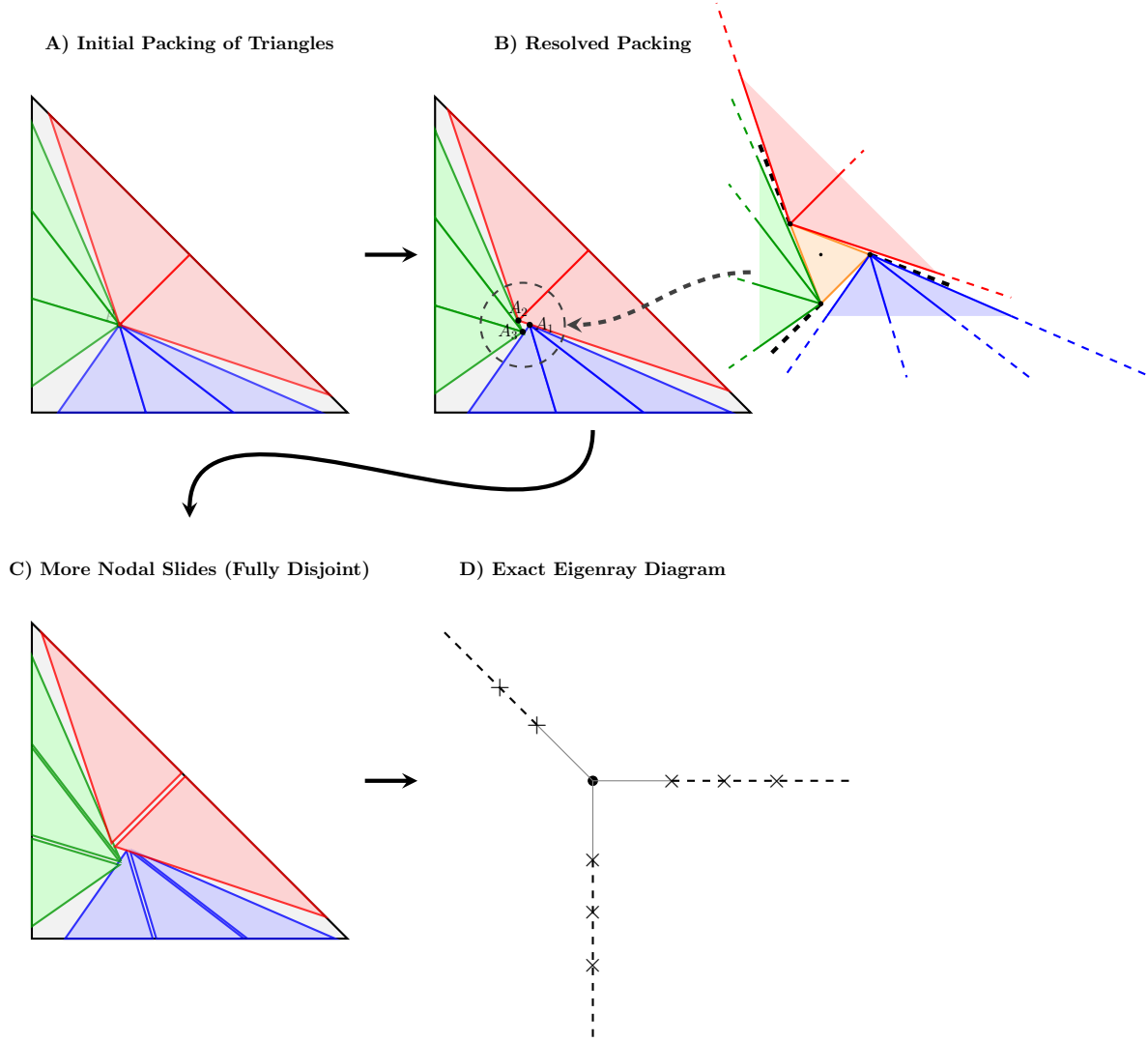
\begin{figure}[htbp]
\centering
\begin{tikzpicture}[scale=0.48, transform shape, >=stealth]
    \begin{scope}[xshift=0cm, yshift=0cm]
        \node[scale=1.4, font=\bfseries] at (4.5, 10.5) {A) Initial Packing of Triangles};
        \draw[thick, fill=gray!10] (0,0) -- (9,0) -- (0,9) -- cycle;
        
        \coordinate (p) at (2.5, 2.5);
        \fill (p) circle (2.5pt) node[above left, font=\Large] {$p$};

        \filldraw[fill=blue!20, draw=blue, thick, opacity=0.7] (p) -- (0.75,0) -- (3.25,0) -- cycle;
        \filldraw[fill=blue!20, draw=blue, thick, opacity=0.7] (p) -- (3.25,0) -- (5.75,0) -- cycle;
        \filldraw[fill=blue!20, draw=blue, thick, opacity=0.7] (p) -- (5.75,0) -- (8.25,0) -- cycle;

        \filldraw[fill=green!20, draw=green!60!black, thick, opacity=0.7] (p) -- (0,8.25) -- (0,5.75) -- cycle;
        \filldraw[fill=green!20, draw=green!60!black, thick, opacity=0.7] (p) -- (0,5.75) -- (0,3.25) -- cycle;
        \filldraw[fill=green!20, draw=green!60!black, thick, opacity=0.7] (p) -- (0,3.25) -- (0,0.75) -- cycle;

        \filldraw[fill=red!20, draw=red, thick, opacity=0.7] (p) -- (8.5,0.5) -- (4.5,4.5) -- cycle;
        \filldraw[fill=red!20, draw=red, thick, opacity=0.7] (p) -- (4.5,4.5) -- (0.5,8.5) -- cycle;
    \end{scope}

    \draw[->, ultra thick] (9.5, 4.5) -- (11.0, 4.5);

    \begin{scope}[xshift=11.5cm, yshift=0cm]
        \node[scale=1.4, font=\bfseries] at (4.5, 10.5) {B) Resolved Packing};
        \draw[thick, fill=gray!10] (0,0) -- (9,0) -- (0,9) -- cycle;

        \coordinate (A1) at (2.7, 2.5); 
        \coordinate (B11L) at (0.95,0); \coordinate (B11R) at (3.45,0);
        \coordinate (B12L) at (3.45,0); \coordinate (B12R) at (5.95,0);
        \coordinate (B13L) at (5.95,0); \coordinate (B13R) at (8.45,0);

        \filldraw[fill=blue!20, draw=blue, thick, opacity=0.8] (A1) -- (B11L) -- (B11R) -- cycle;
        \filldraw[fill=blue!20, draw=blue, thick, opacity=0.8] (A1) -- (B12L) -- (B12R) -- cycle;
        \filldraw[fill=blue!20, draw=blue, thick, opacity=0.8] (A1) -- (B13L) -- (B13R) -- cycle;

        \coordinate (A3) at (2.5, 2.3); 
        \coordinate (B31T) at (0,8.05); \coordinate (B31B) at (0,5.55);
        \coordinate (B32T) at (0,5.55); \coordinate (B32B) at (0,3.05);
        \coordinate (B33T) at (0,3.05); \coordinate (B33B) at (0,0.55);

        \filldraw[fill=green!20, draw=green!60!black, thick, opacity=0.8] (A3) -- (B31T) -- (B31B) -- cycle;
        \filldraw[fill=green!20, draw=green!60!black, thick, opacity=0.8] (A3) -- (B32T) -- (B32B) -- cycle;
        \filldraw[fill=green!20, draw=green!60!black, thick, opacity=0.8] (A3) -- (B33T) -- (B33B) -- cycle;

        \coordinate (A2) at (2.375, 2.625); 
        \coordinate (B21R) at (8.375,0.625); \coordinate (B21L) at (4.375,4.625);
        \coordinate (B22R) at (4.375,4.625); \coordinate (B22L) at (0.375,8.625);

        \filldraw[fill=red!20, draw=red, thick, opacity=0.8] (A2) -- (B21R) -- (B21L) -- cycle;
        \filldraw[fill=red!20, draw=red, thick, opacity=0.8] (A2) -- (B22R) -- (B22L) -- cycle;

        \fill (A1) circle (2.5pt) node[right, font=\Large] {$A_1$};
        \fill (A2) circle (2.5pt) node[above, font=\Large] {$A_2$};
        \fill (A3) circle (2.5pt) node[left, font=\Large] {$A_3$};
        
        \draw[thick, dashed, darkgray] (2.5, 2.5) circle (1.2cm);
        \node[font=\Large, darkgray] at (1.5, 1.5) {};
        
        \coordinate (ZoomRight) at (3.75, 2.5);
    \end{scope}

    \begin{scope}[xshift=22.5cm, yshift=4.5cm]

        \coordinate (ZoomPanelLeft) at (-2.0, -0.5);

        \begin{scope}[scale=7, shift={(-2.5,-2.5)}]
            \coordinate (p) at (2.5, 2.5);
            \coordinate (A1) at (2.7, 2.5); 
            \coordinate (A2) at (2.375, 2.625);
            \coordinate (A3) at (2.5, 2.3);

            \draw[black, ultra thick, dashed] (A1) -- ++(0.325, -0.125); 
            \draw[black, ultra thick, dashed] (A2) -- ++(-0.125, 0.325); 
            \draw[black, ultra thick, dashed] (A3) -- ++(-0.2, -0.2);    

            \draw[orange, thick, fill=orange!20, opacity=0.8] (A1) -- (A2) -- (A3) -- cycle;

            \coordinate (R10) at (2.525, 2.25);
            \coordinate (R11) at (2.775, 2.25);
            \coordinate (R12) at (3.025, 2.25);
            \coordinate (R13) at (3.275, 2.25);
            \fill[blue!20, opacity=0.8] (A1) -- (R10) -- (R13) -- cycle;
            \draw[blue, thick] (A1) -- (R10);
            \draw[blue, thick] (A1) -- (R11);
            \draw[blue, thick] (A1) -- (R12);
            \draw[blue, thick] (A1) -- (R13);
            
            \draw[blue, thick, dashed] (R10) -- ++(-0.175, -0.25);
            \draw[blue, thick, dashed] (R11) -- ++(0.075, -0.25);
            \draw[blue, thick, dashed] (R12) -- ++(0.325, -0.25);
            \draw[blue, thick, dashed] (R13) -- ++(0.575, -0.25);

            \coordinate (R30) at (2.25, 2.875);
            \coordinate (R31) at (2.25, 2.625);
            \coordinate (R32) at (2.25, 2.375);
            \coordinate (R33) at (2.25, 2.125);
            \fill[green!20, opacity=0.8] (A3) -- (R30) -- (R33) -- cycle;
            \draw[green!60!black, thick] (A3) -- (R30);
            \draw[green!60!black, thick] (A3) -- (R31);
            \draw[green!60!black, thick] (A3) -- (R32);
            \draw[green!60!black, thick] (A3) -- (R33);
            
            \draw[green!60!black, thick, dashed] (R30) -- ++(-0.125, 0.2875);
            \draw[green!60!black, thick, dashed] (R31) -- ++(-0.125, 0.1625);
            \draw[green!60!black, thick, dashed] (R32) -- ++(-0.125, 0.0375);
            \draw[green!60!black, thick, dashed] (R33) -- ++(-0.125, -0.0875);

            \coordinate (R20) at (2.975, 2.425);
            \coordinate (R21) at (2.575, 2.825);
            \coordinate (R22) at (2.175, 3.225);
            \fill[red!20, opacity=0.8] (A2) -- (R20) -- (R22) -- cycle;
            \draw[red, thick] (A2) -- (R20);
            \draw[red, thick] (A2) -- (R21);
            \draw[red, thick] (A2) -- (R22);
            
            \draw[red, thick, dashed] (R20) -- ++(0.3, -0.1);
            \draw[red, thick, dashed] (R21) -- ++(0.1, 0.1);
            \draw[red, thick, dashed] (R22) -- ++(-0.1, 0.3);

            \fill (p) circle (0.2pt) node[right, font=\small, yshift=-3] {};
            \fill (A1) circle (0.3pt) node[right, font=\small] {};
            \fill (A2) circle (0.3pt) node[above left, font=\small] {};
            \fill (A3) circle (0.3pt) node[below left, font=\small] {};
        \end{scope}
    \end{scope}

    \draw[->, ultra thick, darkgray, dashed] (ZoomPanelLeft) to[out=180, in=0] (ZoomRight);

    \draw[->, ultra thick] (16.0, -0.5) to[out=270, in=90] (4.5, -3.0);

    \begin{scope}[xshift=0cm, yshift=-15cm]
        \node[scale=1.4, font=\bfseries] at (4.5, 10.5) {C) More Nodal Slides (Fully Disjoint)};
        \draw[thick, fill=gray!10] (0,0) -- (9,0) -- (0,9) -- cycle;

        \filldraw[fill=blue!20, draw=blue, thick, opacity=0.8] (2.7, 2.5) -- (0.95,0) -- (3.45,0) -- cycle;
        \filldraw[fill=blue!20, draw=blue, thick, opacity=0.8] (2.8, 2.5) -- (3.55,0) -- (6.05,0) -- cycle;
        \filldraw[fill=blue!20, draw=blue, thick, opacity=0.8] (2.9, 2.5) -- (6.15,0) -- (8.65,0) -- cycle;

        \filldraw[fill=green!20, draw=green!60!black, thick, opacity=0.8] (2.5, 2.3) -- (0,8.05) -- (0,5.55) -- cycle;
        \filldraw[fill=green!20, draw=green!60!black, thick, opacity=0.8] (2.5, 2.2) -- (0,5.45) -- (0,2.95) -- cycle;
        \filldraw[fill=green!20, draw=green!60!black, thick, opacity=0.8] (2.5, 2.1) -- (0,2.85) -- (0,0.35) -- cycle;

        \filldraw[fill=red!20, draw=red, thick, opacity=0.8] (2.375, 2.625) -- (8.375,0.625) -- (4.375,4.625) -- cycle;
        \filldraw[fill=red!20, draw=red, thick, opacity=0.8] (2.275, 2.725) -- (4.275,4.725) -- (0.275,8.725) -- cycle;
    \end{scope}

    \draw[->, ultra thick] (9.5, -10.5) -- (11.0, -10.5);

    \begin{scope}[xshift=11.5cm, yshift=-15cm]
        \node[scale=1.4, font=\bfseries] at (4.5, 10.5) {D) Exact Eigenray Diagram};
        
        \begin{scope}[shift={(4.5, 4.5)}, scale=1.5]
            \filldraw[black] (0,0) circle (2.5pt);
            
            \begin{scope}[rotate=0]
                \draw[thin, gray] (0,0) -- (1.5,0);
                \node at (1.5, 0) {\Large $\times$};
                \node at (2.5, 0) {\Large $\times$};
                \node at (3.5, 0) {\Large $\times$};
                \draw[thick, dashed] (1.5, 0) -- (5.0, 0);
            \end{scope}

            \begin{scope}[rotate=-90]
                \draw[thin, gray] (0,0) -- (1.5,0);
                \node at (1.5, 0) {\Large $\times$};
                \node at (2.5, 0) {\Large $\times$};
                \node at (3.5, 0) {\Large $\times$};
                \draw[thick, dashed] (1.5, 0) -- (5.0, 0);
            \end{scope}

            \begin{scope}[rotate=135]
                \draw[thin, gray] (0,0) -- (1.5,0);
                \node at (1.5, 0) {\Large $\times$};
                \node at (2.5, 0) {\Large $\times$};
                \draw[thick, dashed] (1.5, 0) -- (4.0, 0);
            \end{scope}
        \end{scope}
    \end{scope}

\end{tikzpicture}
\caption{Consider $U$ obtained from $\mathbb{CP}^2$ by $(3,3,2)$ non-toric blow-ups. We produce an associated rational ray diagram step-by-step. We consider the triangle with vertices at $(0,0),(18,0),(0,18)$ and take $p=(5, 5).$ A) An initial placement of the big triangles intersecting at $p$. B) We apply the Resolving Triangle Packing Lemma. Notice in the zoomed figure here that the line segments from $p$ to $A_i$ are parallel to the edges of the moment triangle. C) Secondary intra-edge nodal slides separate the individual Symington triangles, rendering them pairwise disjoint. D) We apply the branch moves to C and ignore the boundary.}
\label{fig:332}
\end{figure}

In the next step, we use nodal slides in the direction of $\vec{v}_i$ to separate each big triangle into $k_i$ standard Symington triangles (and forget about the line segment if $k_i=0$). Since the big triangles are already disjoint we can easily make sure that all $\sum k_i$ Symington triangles are disjoint.

The final step is to regard these triangles as prescriptions for a non-toric blow-up surgery and do branch moves that make the branch cuts of the chart along the eigenray in the direction of $\vec{v}_i$ for each $i$. In other words, we apply the procedure described in \cite[Example 9.3]{Eva23} at all edges in some order. The result is a compact almost toric base and taking the naive completion of the interior gives an exact eigenray diagram. Leaving out the locations of the nodes, we obtain a rational ray diagram.

We call a rational ray diagram $\mathcal{R}$ obtained from $U$ through this procedure an associated rational ray diagram. Note that the visible two-dimensional fan in this rational ray diagram is isomorphic to the fan of the Toric Start in the procedure. The isomorphism is given by the cross product pairing. The number of nodes on each ray is nothing but the number of non-toric blowups in the corresponding toric divisor.

\FloatBarrier
\subsubsection{Case 2.}

In Case 2(1), we can do nodal trades at each vertex of a square. This leads to the eigenray diagram with rays along the primitive vectors $(1,1),(1,-1),(-1,1), (-1,-1)$ and one node on each of them.

In Case 2(2), the specific Vianna triangle utilized corresponds to the Markov triple $(29, 5, 2)$. This triangle is obtained via a sequence of branch moves applied to the standard Delzant $\mathbb{CP}^2$ moment polygon, characterized by the initial Markov triple $(1, 1, 1)$. We apply three consecutive Markov mutations:
\begin{enumerate}
    \item $(1, 1, 1) \to (2, 1, 1)$: Mutate the weight $1 \to 3(1)(1) - 1 = 2$.
    \item $(2, 1, 1) \to (5, 2, 1)$: Mutate the weight $1 \to 3(2)(1) - 1 = 5$.
    \item $(5, 2, 1) \to (29, 5, 2)$: Mutate the weight $1 \to 3(5)(2) - 1 = 29$.
\end{enumerate}

To embed exactly 8 standard Symington triangles disjointly, the affine length of the target boundary edge must be strictly greater than 8 times the affine distance of the monotone point to the edges (which we normalize to be $1$). For a Vianna triangle, the affine length $L_a$ of the edge corresponding to the maximum weight $a^2$ is exactly $L_a = \frac{3a}{bc}$. Evaluating this length along the mutation sequence yields:
\begin{itemize}
    \item $(1, 1, 1) \implies L = 3 \le 8$
    \item $(2, 1, 1) \implies L = 6 \le 8$
    \item $(5, 2, 1) \implies L = 7.5 \le 8$
    \item $(29, 5, 2) \implies L = 8.7 > 8$
\end{itemize}
The $(29, 5, 2)$ configuration is the first in the sequence satisfying the condition $L_a > 8$.

We choose an explicit model of this integral affine triangle, where the long edge lies on the horizontal axis, the primitive integral direction vectors $U_i$ of the edges (oriented counterclockwise) for the $(29, 5, 2)$ Vianna triangle are:
\begin{itemize}
    \item $U_1 = (1, 0)$ (Long edge, weight $a^2 = 841$)
    \item $U_2 = (-1, 4)$ (Middle edge, weight $b^2 = 25$)
    \item $U_3 = (-204, -25)$ (Short edge, weight $c^2 = 4$)
\end{itemize}Thus the corresponding oriented edge vectors are
$a^2U_1$, $b^2U_2$, and $c^2U_3$, respectively.
These satisfy the required wedge product area balances $U_1 \wedge U_2 = c^2$, $U_2 \wedge U_3 = a^2$, and $U_3 \wedge U_1 = b^2$. We scale the triangle so that the monotone point $p$ has $y$-coordinate $1$.

We compute the primitive vectors along the eigenrays to the three vertices from $p$:
\begin{itemize}
    \item $v_{\text{top}} = (7, 1)$
    \item $v_{\text{left}} = (-41, -5)$
    \item $v_{\text{right}} = (1, -2)$
\end{itemize}

The $k = 9-d$ Symington nodes ($d$ is for degree here) are embedded such that their eigenlines are parallel to the long edge and pass through the monotone point (see Figure \ref{fig-Vianna}). To construct the exact eigenray diagram, $v_{\text{right}}$ must be mutated across these $k$ parallel cuts. Applying the unipotent affine shear transformation in the lower half-plane yields:
\begin{equation}
    v_{\text{right}}' = \begin{pmatrix} 1 & k \\ 0 & 1 \end{pmatrix} \begin{pmatrix} 1 \\ -2 \end{pmatrix} = (1 - 2k, -2)
\end{equation}
Substituting $k = 9 - d$:
\begin{equation}
    v_{\text{right}}' = (1 - 2(9 - d), -2) = (2d - 17, -2)
\end{equation}

As a result the associated exact eigenray diagram has rays along  $(1, 0)$, $(7, 1)$, $(-41, -5)$, and $(2d - 17, -2)$ with $9-d,1,1$ and $1$ nodes respectively. As in Case 1, we call each such rational ray diagram an associated rational ray diagram.


\begin{figure}[p!]
    \centering


    \begin{tikzpicture}[
        scale=0.73,
        xscale=1.6,
        yscale=2.2
    ]

        \coordinate (V1) at (0,0);
        \coordinate (V2) at (8.7,0);
        \coordinate (V3) at (1.5,3.5);

        \fill[gray!10] (V1) -- (V2) -- (V3) -- cycle;

        \draw[very thick]
            (V1) -- (V2)
            node[midway,below=4pt]
            {Long Edge (weight $a^2$, length $8.7$)};

        \draw[thick]
            (V1) -- (V3)
            node[midway,above left=0pt]
            {Short Edge ($c^2$)};

        \draw[thick]
            (V2) -- (V3)
            node[midway,above right=0pt]
            {Middle Edge ($b^2$)};

        \coordinate (P) at (barycentric cs:V1=1,V2=1,V3=1);

        \fill (P) circle (1.5pt)
            node[above=4pt] {Monotone point $p$};

        \draw[dashed,thick,->]
            (P) -- (6.2,1.166)
            node[above right] {Ray parallel to long edge};

        \foreach \i/\xb/\xt/\xtip in {
            1/0.10/1.05/3.45,
            2/1.15/2.10/3.80,
            3/2.20/3.15/4.15,
            4/3.25/4.20/4.50,
            5/4.30/5.25/4.85,
            6/5.35/6.30/5.20,
            7/6.40/7.35/5.55,
            8/7.45/8.40/5.90
        } {
            \draw[
                thick,
                fill=blue!10,
                draw=blue!80!black
            ]
                (\xb,0) --
                (\xt,0) --
                (\xtip,1.166) --
                cycle;

            \draw[thick,blue!80!black]
                (\xtip-0.05,1.166-0.05) --
                (\xtip+0.05,1.166+0.05);

            \draw[thick,blue!80!black]
                (\xtip-0.05,1.166+0.05) --
                (\xtip+0.05,1.166-0.05);

            \pgfmathsetmacro{\xmid}{(\xb+\xt)/2}

            \node[above=1pt,font=\small]
                at (\xmid,0) {$T_{\i}$};
        }

        \draw[
            decorate,
            decoration={brace,amplitude=4pt,mirror},
            yshift=-2pt
        ]
            (0.10,0) -- (1.05,0)
            node[black,midway,below=4pt] {length $1$};

        \coordinate (N1) at ($(V1)!0.06!(P)$);

        \draw[thick,red!80!black]
            (N1) ++(-0.06,-0.06) -- ++(0.12,0.12);

        \draw[thick,red!80!black]
            (N1) ++(-0.06,0.06) -- ++(0.12,-0.12);

        \draw[dashed,thick,red!80!black]
            (N1) -- (V1);

        \node[left=2pt,font=\small,red!80!black]
            at (N1) {Nodal Trade};

        \coordinate (N2) at ($(V2)!0.04!(P)$);

        \draw[thick,red!80!black]
            (N2) ++(-0.06,-0.06) -- ++(0.12,0.12);

        \draw[thick,red!80!black]
            (N2) ++(-0.06,0.06) -- ++(0.12,-0.12);

        \draw[dashed,thick,red!80!black]
            (N2) -- (V2);

        \node[below left=0pt,font=\small,red!80!black]
            at (N2) {Nodal Trade};

        \coordinate (N3) at ($(V3)!0.07!(P)$);

        \draw[thick,red!80!black]
            (N3) ++(-0.06,-0.06) -- ++(0.12,0.12);

        \draw[thick,red!80!black]
            (N3) ++(-0.06,0.06) -- ++(0.12,-0.12);

        \draw[dashed,thick,red!80!black]
            (N3) -- (V3);

        \node[below right=0pt,font=\small,red!80!black]
            at (N3) {Nodal Trade};

    \end{tikzpicture}

    \caption{Schematic of the $(29,5,2)$ Vianna triangle demonstrating
    the geometric capacity to embed eight standard Symington triangles
    disjointly along the long horizontal edge. All lengths are affine,
    and we normalize the affine distance from $p$ to each edge to be
    $1$. Below, we choose a model of the Vianna triangle with
    counterclockwise-oriented edge vectors
    $841(1,0)$, $25(-1,4)$, and $4(-204,-25)$, which is metrically
    different from this schematic picture.}
    \label{fig-Vianna}

    \vspace{2pt}


    \begin{tikzpicture}[
        scale=0.80,
        xscale=0.25,
        yscale=1.2
    ]

        \draw[
            xstep=1,
            ystep=1,
            gray!40,
            very thin
        ]
            (-45,-6) grid (15,3);

        \draw[thick,->]
            (-45,0) -- (18,0)
            node[right] {$x$};

        \draw[thick,->]
            (0,-6) -- (0,3.5)
            node[above] {$y$};

        \fill (0,0) circle (3pt);

\foreach \d in {1,2,3,4,5,6,7,8} {
    \pgfmathtruncatemacro{\xcut}{9-\d}

    \draw[
        very thick,
        orange!85!black,
        ->,
        opacity=0.75
    ]
        (0,0) -- (\xcut,0);

}

\fill[orange!85!black]
    (0,0) circle (2pt);

\node[
    orange!85!black,
    below right,
    font=\scriptsize
]
    at (0,0)
    {};

\node[
    orange!85!black,
    above right
]
    at (7,-0.65)
    {$v_{\mathrm{cut}}(d)=(9-d,0)$};

        \draw[very thick,->]
            (0,0) -- (7,1)
            node[above left]
            {$v_{\mathrm{top}}=(7,1)$};

        \draw[very thick,->]
            (0,0) -- (-41,-5)
            node[below left]
            {$v_{\mathrm{left}}=(-41,-5)$};

        \foreach \d in {1,2,3,4,5,6,7,8,9} {
            \pgfmathtruncatemacro{\xval}{2*\d-17}

            \draw[very thick,blue,->]
                (0,0) -- (\xval,-2);

            \node[
                below,
                blue,
                font=\small,
                rotate=-75,
                anchor=west
            ]
                at (\xval,-2.1)
                {$d=\d:(\xval,-2)$};
        }

        \node[blue,right]
            at (1,-5.5)
            {$v_{\mathrm{right}}'(d)=(2d-17,-2)$};

    \end{tikzpicture}

    \caption{Primitive vectors multiplied by their multiplicities in
    the rational ray diagrams parametrized by $d$, constructed with
    respect to the monotone point of the $(29,5,2)$ Vianna triangle.
    The ray $(9-d,0)$ represents the primitive ray $(1,0)$ with
    multiplicity $9-d$. When $d=9$ this ray disappears.}
    \label{fig-fan}

\end{figure}

\FloatBarrier

\subsection{Rational ray diagram models the Liouville manifold $\widehat{U}$}

We now prove the result that, indeed, the rational ray diagram captures the A-side geometric data of $U$ fully.

We recall the following terminology from \cite{LMN25,LMN23} in order to apply the symplectic Torelli theorem \cite[Theorem 13]{LMN23}.

\begin{definition}
Let $(X,\omega)$ be a closed symplectic four-manifold. A
\emph{symplectic log Calabi--Yau divisor} in $(X,\omega)$ is a
connected union
\[
D=\bigcup_{i=1}^m D_i
\]
of embedded symplectic surfaces such that:
\begin{enumerate}
    \item all intersections between distinct components are positive and
    transverse;
    \item no three distinct components have a common intersection point; and
    \item
    \[
    \sum_{i=1}^m [D_i]
    =
    \operatorname{PD}\bigl(c_1(X,\omega)\bigr)
    \in H_2(X;\mathbb{Z}).
    \]
\end{enumerate}
Here, $c_1(X,\omega)$ denotes the first Chern class determined by any
$\omega$-tame almost complex structure. The triple $(X,\omega,D)$ is
called a \emph{symplectic log Calabi--Yau pair}.

The pair is called \emph{orthogonal} if, for every
$p\in D_i\cap D_j$ with $i\neq j$, the tangent spaces $T_pD_i$ and
$T_pD_j$ are $\omega$-orthogonal; equivalently,
\[
\omega(v_i,v_j)=0
\]
for all $v_i\in T_pD_i$ and $v_j\in T_pD_j$.
\end{definition}

\begin{definition}
Let $(M,\omega,D)$ and $(M',\omega',D')$ be symplectic log
Calabi--Yau pairs. An \emph{isomorphism of symplectic log
Calabi--Yau pairs} is a diffeomorphism
\[
\Phi\colon M\xrightarrow{\sim}M'
\]
such that
\[
\Phi^*\omega'=\omega
\qquad\text{and}\qquad
\Phi(D)=D'
\]
as stratified symplectic divisors. Thus $\Phi$ induces a bijection
between the irreducible components of $D$ and those of $D'$, and
identifies the local branches and intersection strata at every node.
If the components are labelled, we additionally require $\Phi$ to
preserve their labels.
\end{definition}

\begin{lemma}\label{lem-preferred-grading}
Let $(M,\omega,D)$ be a symplectic log Calabi--Yau pair and suppose
that
\[
H^1(M;\mathbb Z)=0.
\]
Then $M\setminus D$ carries a preferred homotopy class of
trivializations of its canonical bundle, and hence a preferred grading
structure. Moreover, every isomorphism of such symplectic log
Calabi--Yau pairs preserves these preferred grading structures.
\end{lemma}

\begin{proof}
Write
\[
D=D_1\cup\cdots\cup D_k
\]
and let
\[
\mathcal O(D):=
\mathcal O(D_1)\otimes\cdots\otimes\mathcal O(D_k)
\]
be the smooth complex line bundle associated to the oriented
codimension-two divisor $D$. It carries a canonical section $s_D$,
well defined up to homotopy, whose zero set is $D$, with multiplicity
one along each component.

Choose an $\omega$-compatible almost complex structure and denote its
canonical bundle by $K_M$. Since $(M,\omega,D)$ is a symplectic log
Calabi--Yau pair,
\[
c_1(K_M^{-1})=c_1(TM)=\operatorname{PD}[D]
=c_1\bigl(\mathcal O(D)\bigr).
\]
Consequently, there exists an isomorphism of complex line bundles
\[
\alpha\colon\mathcal O(D)\xrightarrow{\sim}K_M^{-1}.
\]
The section $\alpha(s_D)$ is nowhere vanishing on $M\setminus D$ and
therefore determines a trivialization of $K_M^{-1}$ there, or,
equivalently, a trivialization of $K_M|_{M\setminus D}$.

Any two choices of $\alpha$ differ by a map
\[
M\longrightarrow\mathbb C^*.
\]
The homotopy classes of such maps form $H^1(M;\mathbb Z)$, which
vanishes by assumption. The resulting homotopy class of
trivializations is therefore independent of $\alpha$. It is also
independent of the compatible almost complex structure, since the
space of compatible almost complex structures is contractible. This
defines the preferred grading structure.

If
\[
\Phi\colon(M,\omega,D)\xrightarrow{\sim}
(M',\omega',D')
\]
is an isomorphism of symplectic log Calabi--Yau pairs, then $\Phi$
identifies $\mathcal O(D)$ with $\Phi^*\mathcal O(D')$, including
their canonical sections up to homotopy, and, choosing an $\omega'$-compatible almost complex structure $J'$ on $M'$
and equipping $M$ with the compatible almost complex structure
$J=\Phi^*J'$, its differential
identifies $K_M^{-1}$ with $\Phi^*K_{M'}^{-1}$. Thus, the pullback by
$\Phi$ of the preferred grading on $M'\setminus D'$ is one of the
grading structures constructed above on $M\setminus D$. By uniqueness,
it is the preferred grading structure.
\end{proof}

We now restate Theorem \ref{thm-ray} in light of the previous sections and prove it.

\begin{theorem} Let $U$ be an affine log CY surface and $\mathcal{R}$ be a rational ray diagram associated to $U$ (meaning that it arises from applying the procedure that we described in the previous section). Then $U_\mathcal{R}$ is strongly exact symplectomorphic to $\widehat{U}$ in a way that identifies the grading structures fixed above.
\end{theorem}

\begin{proof}
Let $B_{\mathrm{surg}}$ be the compact almost toric base obtained in the
construction of the associated eigenray diagram. It determines a closed
orthogonal symplectic log Calabi--Yau pair
\[
(M_{\mathrm{AT}},\omega_{\mathrm{AT}},D_{\mathrm{AT}}).
\]
Consider a generalized toric model of $U$ as in Theorem \ref{thm-toric-model} using the notation introduced prior to it. Choose the K\"ahler form $\omega_{Y,0}$ on the algebraic
compactification $(Y,D'')$ as in Theorem \ref{thm-tri-pack}. The triple $(Y, \omega_{Y,0}, D'')$ is also a symplectic log CY pair, but it need not be orthogonal.

We will now perturb the symplectic form to make the boundary divisor orthogonal, but we need to keep track of the Liouville structures, so let us introduce some notation. Consider the ample divisor supported on the boundary,
\[
A=\sum_i \mu_iC_i'',
\qquad
\mu_i\in\mathbb Z_{>0}.
\]
Let $s_A$ be the
canonical section of $\mathcal O_Y(A)$, and choose a Hermitian metric
$h$ whose Chern curvature is $\omega_{Y,0}$. On $U=Y\setminus D'',$ set
\[
q=\lVert s_A\rVert_h^2,
\qquad
\varphi=-\log q.
\]
Thus
\[
\lambda_{\mathrm{alg}}:=-d^c\varphi
\]
satisfies
\[
d\lambda_{\mathrm{alg}}=\omega_{Y,0}|_U.
\]

By the standard orthogonalization lemma
\cite[Lemma~2.3]{Gompf1995}, followed by isotopy extension and
pullback, there is a family of cohomologous symplectic forms
\[
\{\omega_{Y,s}\}_{s\in[0,1]}
\]
on \(Y\), with \(\omega_{Y,0}\) equal to the original K\"ahler form,
such that \(D''\) is \(\omega_{Y,s}\)-symplectic for every \(s\) and is
\(\omega_{Y,1}\)-orthogonal. We henceforth write
\[
\omega_Y:=\omega_{Y,1}.
\]
Since $[\omega_{Y,s}]=[\omega_{Y,0}]$
for every $s$, this deformation does not change any of the
symplectic-area calculations below.

Choose a smooth family
\(\beta_s\in\Omega^1(Y)\), with \(\beta_0=0\), such that
\[
\omega_{Y,s}=\omega_{Y,0}+d\beta_s,
\]
and set
\[
\theta_s=\lambda_{\mathrm{alg}}+\beta_s|_U.
\]
The forms \(\theta_s\) have the same negative wrapping numbers around
the components of \(D''\). Moreover, since $s_A$ vanishes along $C_i''$ with multiplicity
$\mu_i$, the function $-\varphi=\log\lVert s_A\rVert_h^2$ is
compatible with $D''$ in the sense of McLean
\cite[page~26]{McLean2016}. The negative-wrapping version of McLean's construction
\cite[Proposition~2.5]{LiMakCapping}, applied parametrically as in
\cite[Proposition~5.10]{McLean2016},
provides a smooth family of functions \(g_s\colon U\to\mathbb{R}\)
such that, if \(Z_s\) is the Liouville vector field of
\(\theta_s+dg_s\), then
\[
d\varphi(Z_s)>0
\]
throughout a sufficiently small punctured neighborhood of \(D''\),
uniformly in \(s\). 

Set
\[
\widetilde{\theta}_s:=\theta_s+dg_s.
\]
Since the assertion already holds for $\lambda_{\mathrm{alg}}$ at
$s=0$, the parametrized construction may be performed relative to
$s=0$, and hence we may assume that $g_0=0$. Choose $C$ sufficiently
large that
\[
W_C:=\{\varphi\leq C\}
\]
is a Liouville domain for $\widetilde{\theta}_s$ for every
$s\in[0,1]$. The standard completion equivalence theorem for a
compact Liouville deformation (i.e. \cite[Proposition 11.8]{CE}) then shows that the completions of
$(W_C,\widetilde{\theta}_0)$ and $(W_C,\widetilde{\theta}_1)$ are
strongly exact symplectomorphic. Define
\[
\lambda_Y:=\widetilde{\theta}_1
=\lambda_{\mathrm{alg}}+\beta_1|_U+dg_1.
\]
Thus $d\lambda_Y=\omega_Y|_U$.

We apply the symplectic Torelli theorem \cite[Theorem 13]{LMN23} for orthogonal symplectic log
Calabi--Yau pairs. Let $X$ be the toric starting surface, and let
$E_{i,k}$ denote the exceptional classes of the non-toric blow-ups.
On the almost toric side, let $V_{i,k}$ be the corresponding
exceptional spheres arising from the Symington triangles. There are
marked decompositions
\[
H_2(Y;\mathbb Z)
\cong
H_2(X;\mathbb Z)
\oplus
\bigoplus_{i=1}^{m}\bigoplus_{k=1}^{k_i}
\mathbb Z[E_{i,k}]
\]
and
\[
H_2(M_{\mathrm{AT}};\mathbb Z)
\cong
H_2(X;\mathbb Z)
\oplus
\bigoplus_{i=1}^{m}\bigoplus_{k=1}^{k_i}
\mathbb Z[V_{i,k}].
\]
Define
\[
\Psi\colon
H_2(Y;\mathbb Z)
\longrightarrow
H_2(M_{\mathrm{AT}};\mathbb Z)
\]
to be the identity on $H_2(X;\mathbb Z)$ and to send
$[E_{i,k}]$ to $[V_{i,k}]$. Since the two constructions perform
symplectic blow-ups of the same toric starting surface, $\Psi$
preserves the intersection form. In the normal-crossings case it also
preserves the individual boundary classes:
\[
\Psi([C_i''])
=
\Psi\left(
[C_i]-\sum_{k=1}^{k_i}[E_{i,k}]
\right)
=
[C_i]-\sum_{k=1}^{k_i}[V_{i,k}]
=
[D_{\mathrm{AT},i}].
\]
In the elliptic case, the same calculation after summing the toric
boundary classes gives
\[
\Psi([D''])=[D_{\mathrm{AT}}].
\]

The symplectic areas also agree. The classes coming from $X$ have
areas equal to the affine lengths of the corresponding edges of the
initial moment polygon, while
\[
\int_{E_{i,k}}\omega_Y
=
\mu_i
=
\int_{V_{i,k}}\omega_{\mathrm{AT}}.
\]
Consequently,
\[
\Psi^*[\omega_{\mathrm{AT}}]=[\omega_Y].
\]

The symplectic Torelli theorem therefore gives an isomorphism of symplectic log Calabi-Yau pairs
\[
\Phi\colon
(Y,\omega_Y,D'')
\xrightarrow{\sim}
(M_{\mathrm{AT}},\omega_{\mathrm{AT}},D_{\mathrm{AT}}).
\]

Notice that $\Phi\colon Y\to M_{\mathrm{AT}}$ lifts to a
diffeomorphism between the manifolds with corners obtained by taking
the iterated real oriented blow-ups along the components of the
boundary divisors. 

Put
\[
U_{\mathrm{AT}}
=
M_{\mathrm{AT}}\setminus D_{\mathrm{AT}},
\qquad
q_{\mathrm{pull}}
=
q\circ\Phi^{-1},
\qquad
\lambda_{\mathrm{pull}}
=
(\Phi^{-1})^*\lambda_{\mathrm{Y}}.
\]
Let $\lambda_{\mathrm{AT}}$ and $Z_{\mathrm{AT}}$ denote the
Liouville form and Liouville vector field arising from the radial
Liouville structure of the exact eigenray diagram. We claim that
\[
dq_{\mathrm{pull}}(Z_{\mathrm{AT}})<0
\]
throughout a sufficiently small punctured neighborhood of
$D_{\mathrm{AT}}$.

Consider first a crossing of two components of $D_{\mathrm{AT}}$.
Let $\rho_1,\rho_2$ be the standard action coordinates given by the
affine distances from the two adjacent edges, and write
\[
\rho_i=\frac{r_i^2}{2}.
\]
Since $\Phi$ lifts to the real oriented blow-ups and
$q=\lVert s_A\rVert_h^2$ vanishes to order $2\mu_i$ in a normal
radial coordinate along the $i$th component, we may write
\[
q_{\mathrm{pull}}
=
\rho_1^{\mu_1}\rho_2^{\mu_2}u
\]
near the crossing, where $u$ is positive and smooth on the
corresponding real oriented blow-up.

In these action coordinates, the radial Liouville vector field is
\[
Z_{\mathrm{AT}}
=
(\rho_1-b_1)\frac{\partial}{\partial\rho_1}
+
(\rho_2-b_2)\frac{\partial}{\partial\rho_2},
\]
where
\(
b_i=\rho_i(p)
\)
is the affine distance from the distinguished interior point $p$ to
the $i$th boundary edge. By the choice of the moment polygon in
Theorem~\ref{thm-tri-pack}, this distance is equal to $\mu_i$; hence
\[
b_i=\mu_i.
\]
We retain the notation $b_i$ in the following calculation only to
distinguish its role as a coordinate of $p$ from the role of $\mu_i$
as the vanishing multiplicity of $q_{\mathrm{pull}}$. Equivalently,
\[
Z_{\mathrm{AT}}
=
\left(-\frac{b_1}{r_1}+\frac{r_1}{2}\right)
\frac{\partial}{\partial r_1}
+
\left(-\frac{b_2}{r_2}+\frac{r_2}{2}\right)
\frac{\partial}{\partial r_2}.
\]
Since $u$ is smooth on the real oriented
blow-up, its first partial derivatives with respect to $r_1,r_2$ (and the angle coordinates) are bounded. It follows from the
displayed formula that, for some $C>0$,
\[
\left|Z_{\mathrm{AT}}(u)\right|
\leq
C\left(1+\frac{1}{r_1}+\frac{1}{r_2}\right).
\]

Direct differentiation gives
\[
Z_{\mathrm{AT}}(q_{\mathrm{pull}})
=
\rho_1^{\mu_1-1}\rho_2^{\mu_2-1}B,
\]
where
\[
B
=
\rho_1\rho_2
\bigl(Z_{\mathrm{AT}}(u)+(\mu_1+\mu_2)u\bigr)
-\mu_1b_1\rho_2u-\mu_2b_2\rho_1u.
\]
Since $b_i>0$ and $u$ is positive, after shrinking the neighborhood
there is a constant $c>0$ such that
\[
\mu_i b_i u\geq c,
\qquad i=1,2.
\]
Consequently, after modifying the constants if necessary,
\[
B
\leq
-c(r_1^2+r_2^2)
+
C(r_1^2r_2+r_1r_2^2)
+
C'r_1^2r_2^2
\]
for some $C,C'>0$. If $0<r_1,r_2\leq\delta$, then
\[
B
\leq
\bigl(-c+C\delta+C'\delta^2\bigr)(r_1^2+r_2^2).
\]
Choosing $\delta>0$ sufficiently small therefore gives
\[
dq_{\mathrm{pull}}(Z_{\mathrm{AT}})
=
Z_{\mathrm{AT}}(q_{\mathrm{pull}})<0
\]
near the crossing.

Near a smooth point of the $i$-th boundary component, the same
argument begins with
\[
q_{\mathrm{pull}}=\rho_i^{\mu_i}u
\]
and gives the desired inequality directly. Since
$D_{\mathrm{AT}}$ is compact, these local neighborhoods may be chosen
to cover a sufficiently small punctured neighborhood of the entire
divisor.

On the other hand, since
\[
q_{\mathrm{pull}}
=
\exp\bigl(-\varphi\circ\Phi^{-1}\bigr),
\]
the construction of $\lambda_Y$ gives
\[
dq_{\mathrm{pull}}(Z_{\mathrm{pull}})
=
-q_{\mathrm{pull}}\,
d\bigl(\varphi\circ\Phi^{-1}\bigr)(Z_{\mathrm{pull}})
<0
\]
near $D_{\mathrm{AT}}$.

Hence, by an elementary compactness argument, we can find $\epsilon>0$ sufficiently small so that
\[
W_\epsilon:=\{q_{\mathrm{pull}}\geq\epsilon\}
\] is a common Liouville domain for $\lambda_{\mathrm{pull}}$ and
$\lambda_{\mathrm{AT}}$. Then, for
\[
\Lambda_t
:=(1-t)\lambda_{\mathrm{pull}}+t\lambda_{\mathrm{AT}},
\qquad t\in[0,1],
\]
we have $d\Lambda_t=\omega_{\mathrm{AT}}$, and the corresponding
Liouville vector field is
\[
Z_t=(1-t)Z_{\mathrm{pull}}+tZ_{\mathrm{AT}}.
\]
Consequently,
\[
dq_{\mathrm{pull}}(Z_t)<0
\]
along $\partial W_\epsilon$ for every $t$. Thus
$\{\Lambda_t\}_{t\in[0,1]}$ is a compact Liouville deformation, and
\cite[Proposition~11.8]{CE} gives a strongly exact symplectomorphism
between its completed endpoints. Composing this with the preceding
completion equivalence and with $\Phi$ gives
\[
\widehat U\simeq U_{\mathcal R}
\]
by a strongly exact symplectomorphism.

It remains to discuss the grading. Both $Y$ and $M_{\mathrm{AT}}$
are rational surfaces, and hence
\[
H^1(Y;\mathbb Z)=H^1(M_{\mathrm{AT}};\mathbb Z)=0.
\]
Lemma~\ref{lem-preferred-grading} therefore equips the complements of their boundary divisors
with preferred grading structures.

We first compare these preferred gradings with the grading data fixed
previously. Let $\Omega$ be the non-vanishing holomorphic volume form
on $U$ with simple poles along $D''$. Its inverse extends across $D''$
to an anticanonical section vanishing with multiplicity one along each
component of $D''$. Consequently, the grading datum $[\Omega]$ used
to grade $\widehat U$ is precisely the preferred grading associated
to $(Y,\omega_{Y,0},D'')$. Along the deformation
$\{\omega_{Y,s}\}_{s\in[0,1]}$, these preferred grading structures
form a continuous family, and hence the first completion
symplectomorphism carries this grading to the preferred grading
associated to $(Y,\omega_Y,D'')$.

By Lemma~\ref{lem-preferred-grading}, the Torelli symplectomorphism
\[
\Phi\colon(Y,\omega_Y,D'')
\xrightarrow{\sim}
(M_{\mathrm{AT}},\omega_{\mathrm{AT}},D_{\mathrm{AT}})
\]
carries this grading to the preferred grading on
$U_{\mathrm{AT}}$. This is also the grading on $U_{\mathcal R}$
specified previously. Indeed, near a crossing of the boundary, the
preferred trivialization is homotopic to
\[
\frac{dz_1}{z_1}\wedge\frac{dz_2}{z_2},
\]
which has constant phase on a crossing torus. Therefore such fibers have Maslov class zero. In the elliptic case, the residue framing on $D_{\mathrm{AT}}$
is homotopic to the framing obtained by pulling back
$\eta=\operatorname{Res}_{D''}\Omega$ along
$\Phi^{-1}|_{D_{\mathrm{AT}}}$.
An essential embedded circle in $D_{\mathrm{AT}}$ therefore corresponds
to an essential embedded circle in $D''$. In the flat structure
determined by $\eta$, such a circle is isotopic through embedded circles
to a straight circle, and hence has rotation number zero in the residue
framing. A regular fiber sufficiently close to $D_{\mathrm{AT}}$ is the
product of a normal meridian with such an essential circle, and the
normal meridian has constant phase with respect to the logarithmic
trivialization. Hence this fiber has Maslov class zero. Since the regular
locus of the almost toric base is connected, parallel transport along
paths in the regular locus gives Lagrangian isotopies between regular
fibers. Hence every regular fiber has Maslov class zero in both cases. By the uniqueness statement
following Proposition \ref{prop-R-Liouville-eq}, the preferred grading is precisely the
grading datum previously assigned to $U_{\mathcal R}$.

Finally, the strongly exact symplectomorphisms arising from the two
Liouville deformations are induced by diffeotopies starting at the
identity and therefore preserve the corresponding homotopy classes
of grading structures. Their composition with $\Phi$ gives the
required grading-preserving strongly exact symplectomorphism.

\end{proof}

\appendix
\section{Computation of the Minimal Boundary Cycle}

In this appendix, we explicitly resolve the toric surface constructed
for Case~2(2), apply the non-toric blowup recipe, and perform the
boundary blowdown cascade to prove that the resulting open surface is
isomorphic to the complement of a singular fiber of Kodaira type
$I_d$, where $d=9-k$, in a rational elliptic surface.

\subsection*{A.1. The Fan and Cone Determinants}
The complete toric fan $\Sigma_k$ in $N_\mathbb{R} \cong \mathbb{R}^2$ is generated by the four primitive rays derived from the rational ray diagram:
\begin{align*}
    R_1 &= (1, 0) \\
    R_2 &= (7, 1) \\
    R_3 &= (-41, -5) \\
    R_4 &= (1 - 2k, -2)
\end{align*}
Arranged in counterclockwise order, the four 2-dimensional cones $\sigma_{ij} = \text{Cone}(R_i, R_j)$ have determinants:
\begin{itemize}
    \item $\Delta_{12} = \det(R_1, R_2) = (1)(1) - (0)(7) = 1$
    \item $\Delta_{23} = \det(R_2, R_3) = (7)(-5) - (1)(-41) = -35 + 41 = 6$
    \item $\Delta_{34} = \det(R_3, R_4) = (-41)(-2) - (-5)(1-2k) = 82 + 5 - 10k = 87 - 10k$
    \item $\Delta_{41} = \det(R_4, R_1) = (1-2k)(0) - (-2)(1) = 2$
\end{itemize}

\subsection*{A.2. Hirzebruch-Jung Resolution of Toric Singularities}
We resolve the singular cones by inserting interior rays such that every adjacent pair has determinant $1$, yielding a smooth toric surface $\widetilde{Y}_{\Sigma_k}$. The self-intersections of the corresponding boundary curves are denoted by $C_{r_i}^2 = -b_i$, where $r_{i-1} + r_{i+1} = b_i r_i$.

\paragraph{Cone 12 ($\Delta_{12} = 1$):} 
Smooth. No interior rays are needed.

\paragraph{Cone 41 ($\Delta_{41} = 2$):} 
Requires 1 interior ray $T_1$:
\begin{equation}
    T_1 = (1-k, -1)
\end{equation}
Toric self-intersection: $R_4 + R_1 = (2-2k, -2) = 2 T_1 \implies C_{T_1}^2 = -2$.

\paragraph{Cone 23 ($\Delta_{23} = 6$):} 
The Hirzebruch-Jung continued fraction $\frac{6}{5} = [2,2,2,2,2]$ yields 5 interior rays $S_1, \dots, S_5$:
\begin{equation}
    S_1 = (-1, 0), \quad S_2 = (-9, -1), \quad S_3 = (-17, -2), \quad S_4 = (-25, -3), \quad S_5 = (-33, -4)
\end{equation}
Every $S_i$ satisfies $S_{i-1} + S_{i+1} = 2 S_i \implies C_{S_i}^2 = -2$.

\paragraph{Cone 34 ($\Delta_{34} = 87 - 10k$):} 
For $k < 8$, the boundary of the convex hull of integer points generates a sequence of $9-k$ intermediate rays $W_j$ strictly along the line $y = -1$:
\begin{equation}
    W_j = (j-9, -1) \quad \text{for } j \in \{1, \dots, 9-k\}
\end{equation}
Toric self-intersections:
\begin{itemize}
    \item At $W_1$: $R_3 + W_2 = (-41, -5) + (-7, -1) = (-48, -6) = 6 W_1 \implies C_{W_1}^2 = -6$.
    \item At internal $W_j$: $W_{j-1} + W_{j+1} = 2 W_j \implies C_{W_j}^2 = -2$.
    \item At $W_{9-k}$: $W_{9-k-1} + R_4 = (-k-1, -1) + (1-2k, -2) = (-3k, -3) = 3 W_{9-k} \implies C_{W_{9-k}}^2 = -3$.
\end{itemize}
For the case $k=8$, there is exactly 1 interior ray $W_1 = (-8, -1)$ with $C_{W_1}^2 = -7$.

\paragraph{Main Rays:}
\begin{itemize}
    \item $R_1$: $T_1 + R_2 = (1-k, -1) + (7, 1) = (8-k, 0) = (8-k) R_1 \implies C_{R_1}^2 = -(8-k) = k-8$.
    \item $R_2$: $R_1 + S_1 = (1, 0) + (-1, 0) = (0, 0) = 0 \cdot R_2 \implies C_{R_2}^2 = 0$.
    \item $R_3$: $S_5 + W_1 = (-33, -4) + (-8, -1) = (-41, -5) = 1 \cdot R_3 \implies C_{R_3}^2 = -1$.
    \item $R_4$: $W_{9-k} + T_1 = (-k, -1) + (1-k, -1) = (1-2k, -2) = 1 \cdot R_4 \implies C_{R_4}^2 = -1$.
\end{itemize}

\subsection*{A.3. Non-Toric Blowups}
We apply $k$ non-toric blowups on the branch cut ray $R_1$ and $1$ blowup on each of $R_2, R_3,$ and $R_4$. This reduces their self-intersections to:
\begin{itemize}
    \item $R_1: (k - 8) - k = -8$
    \item $R_2: 0 - 1 = -1$
    \item $R_3: -1 - 1 = -2$
    \item $R_4: -1 - 1 = -2$
\end{itemize}
For $k<8$, the full cyclic boundary array of self-intersections is:
\begin{equation}
    [\mathbf{R_1(-8)}, \mathbf{R_2(-1)}, S_1(-2), \dots, S_5(-2), \mathbf{R_3(-2)}, W_1(-6), \dots, W_{9-k}(-3), \mathbf{R_4(-2)}, T_1(-2)]
\end{equation} For $k=8,$ the $W$ chain contains only one curve of self-intersection $-7.$

\subsection*{A.4. Iterative Blowdown Cascade}
We iteratively contract boundary $(-1)$-curves. When $C_i^2 = -1$ is contracted, its neighbors increment: $C_{i-1}^2 \mapsto C_{i-1}^2 + 1$ and $C_{i+1}^2 \mapsto C_{i+1}^2 + 1$. We go through the process for $k<8$ first.

\paragraph{Phase 1: Forward cascade through the $S$-sequence.}
The single $(-1)$-curve is initially at $R_2$.
\begin{enumerate}
    \item Contract $R_2(-1)$: $R_1 \to -7$, $S_1 \to -1$.
    \item Contract $S_1(-1)$: $R_1 \to -6$, $S_2 \to -1$.
    \item Contract $S_2(-1)$: $R_1 \to -5$, $S_3 \to -1$.
    \item Contract $S_3(-1)$: $R_1 \to -4$, $S_4 \to -1$.
    \item Contract $S_4(-1)$: $R_1 \to -3$, $S_5 \to -1$.
    \item Contract $S_5(-1)$: $R_1 \to -2$, $R_3 \to -1$.
    \item Contract $R_3(-1)$: $R_1 \to \mathbf{-1}$, $W_1 \to \mathbf{-5}$.
\end{enumerate}
The surviving sequence is now:
\begin{equation}
    [\mathbf{R_1(-1)}, \mathbf{W_1(-5)}, W_2(-2), \dots, W_{9-k}(-3), R_4(-2), \mathbf{T_1(-2)}]
\end{equation}

\paragraph{Phase 2: $R_1$ Contraction.}
\begin{enumerate}\setcounter{enumi}{7}
    \item Contract $R_1(-1)$: Neighbors $T_1$ and $W_1$ increment:
    $$T_1 \to \mathbf{-1}, \quad W_1 \to -5 + 1 = \mathbf{-4}$$
\end{enumerate}
The sequence is:
\begin{equation}
    [\mathbf{W_1(-4)}, W_2(-2), \dots, \mathbf{W_{9-k}(-3)}, \mathbf{R_4(-2)}, \mathbf{T_1(-1)}]
\end{equation}

\paragraph{Phase 3: Backward Wrap-Around.}
\begin{enumerate}\setcounter{enumi}{8}
    \item Contract $T_1(-1)$: Neighbors $R_4$ and $W_1$ increment:
    $$R_4 \to -2 + 1 = \mathbf{-1}, \quad W_1 \to -4 + 1 = \mathbf{-3}$$
    \item Contract $R_4(-1)$: Neighbors $W_{9-k}$ and $W_1$ increment. 
    For $k < 8$, where $W_{9-k} \neq W_1$, both curves increment independently:
    $$W_{9-k} \to -3 + 1 = \mathbf{-2}, \quad W_1 \to -3 + 1 = \mathbf{-2}$$
\end{enumerate}

 All $(-1)$-curves have been eliminated, and the cascade halts. The only surviving curves are exactly the $W$-sequence:
\begin{equation}
    W_1, W_2, \dots, W_{9-k}
\end{equation}

 For \(k=8\), the same process brings us after Step~9 to the
2-cycle \([W_1(-4),R_4(-1)]\). Here \(W_1\cdot R_4=2\). Contracting
\(R_4\) therefore changes the self-intersection of \(W_1\) to
\[
-4+(W_1\cdot R_4)^2=-4+4=0,
\]
and identifies the two intersection points. Thus the image of \(W_1\)
is an irreducible nodal rational curve of self-intersection \(0\), namely
an \(I_1\)-fiber.

\subsection*{A.5. The period calculation.}

The self-intersection computation above determines the dual graph of the
boundary, but this alone does not imply that the resulting pair is a
rational elliptic surface: a cycle of $(-2)$-curves may have a nontrivial
period.  We therefore use the following additional input: for the anticanonical divisor $D$ obtained after the
blowdown cascade, the period is trivial, equivalently
\[
\mathcal{O}_{D}(D)\cong \mathcal{O}_{D}.
\]
We recall briefly the period invariant of a Looijenga pair.  Let
\[
D=D_1+\cdots+D_n
\]
be an oriented cycle of rational curves and set
\[
\Lambda_Y
=
\bigl\{
L\in\operatorname{Pic}(Y)
\;\big|\;
L\cdot D_i=0\text{ for every }i
\bigr\}.
\]
The orientation of the cycle gives a canonical identification
\[
\operatorname{Pic}^0(D)\cong\mathbb G_m.
\]
More explicitly, if \(q_i=D_i\cap D_{i+1}\) and
\(L\in\operatorname{Pic}^0(D)\), choose a nonzero trivializing section
\(\sigma_i\) of \(L|_{D_i}\).  The element
\[
\lambda(L)
=
\prod_i
\frac{\sigma_{i+1}(q_i)}{\sigma_i(q_i)}
\in\mathbb G_m
\]
is independent of the chosen trivializations and identifies
\(\operatorname{Pic}^0(D)\) with \(\mathbb G_m\).  In particular,
\(L\cong\mathcal O_D\) if and only if \(\lambda(L)=1\).  For an
irreducible nodal rational curve the same definition is made using the
two inverse images of the node in its normalization.

The period homomorphism of the pair is
\[
\varphi_Y:\Lambda_Y\longrightarrow\operatorname{Pic}^0(D),
\qquad
L\longmapsto L|_D.
\]
Thus the self-intersection sequence determines only the multidegree of
\(L|_D\), whereas the period records the remaining gluing parameter.
In particular, if \(D\cdot D_i=0\) for every \(i\), then
\[
\varphi_Y([D])=\mathcal O_D(D).
\]

\begin{lemma}
For the Looijenga pair constructed above by performing all non-toric
blow-ups at the points \(-1\), the period homomorphism is trivial.
Consequently,
\[
\mathcal O_D(D)\cong\mathcal O_D
\]
for the boundary obtained after the blowdown cascade.
\end{lemma}

\begin{proof}
By \cite[Lemma~2.8(1) and Proposition~2.9]{GHKModuli}, the marking
of the toric boundary by the points corresponding to $-1$ has trivial
marked period, and this remains true after any sequence of possibly
infinitely near non-toric blow-ups centered at these points. Restricting
the marked period to $\Lambda_Y$ shows that the ordinary period
homomorphism is trivial. Since the period homomorphism is invariant under
corner blow-ups and corner blow-downs, the same holds after the boundary
blowdown cascade.
\end{proof}

\subsection*{A.6. Conclusion}

Since \(Y\) is rational, \(H^1(Y,\mathcal O_Y)=0\).  Applying
cohomology to
\[
0\longrightarrow\mathcal O_Y
\longrightarrow\mathcal O_Y(D)
\longrightarrow\mathcal O_D(D)
\longrightarrow 0
\]
and using \(\mathcal O_D(D)\cong\mathcal O_D\) gives
\[
h^0(Y,\mathcal O_Y(D))=2.
\]
Let \(s_D\) be the section cutting out \(D\), and let \(s_1\) be a
section whose restriction to \(D\) is \(1\).  The sections \(s_D\) and
\(s_1\) have no common zero: \(s_D\) is nonzero away from \(D\), while
\(s_1\) is nonzero along \(D\).  Hence
\[
|D|=|-K_Y|
\]
is a basepoint-free pencil.  Adjunction shows that its general member
has genus one.  Moreover, since \(Y\) is rational and \(K_Y^2=0\), it
contains a \((-1)\)-curve \(E\); adjunction gives
\[
D\cdot E=-K_Y\cdot E=1,
\]
so \(E\) is a section of the pencil.  Moreover, every $(-1)$-curve $E'\subset Y$ satisfies
$D\cdot E'=1$, so no $(-1)$-curve is contained in a fiber.
Consequently, the elliptic fibration is relatively minimal. Thus $Y$ is a rational elliptic surface and $D$ is a singular fiber
of Kodaira type $I_{9-k}$. Consequently, the resulting open surface
is the complement of an $I_{9-k}$-fiber in a rational elliptic
surface, for every $k\in\{0,\ldots,8\}$.

\end{document}